\documentclass[12pt]{amsart}  % draft

\usepackage{custom_preamble} % showkeys

\begin{document}

\title{Boolean functions with small spectral norm, revisited}

\author{\tsname}
\address{\tsaddress}
\email{\tsemail}

\maketitle

\section{Introduction}

The purpose of this note is to present the argument from \cite{san::12} for groups of the form $\F_2^n$.  Throughout $G$ will denote such a group; the whole point of our arguments is that the results will not depend on $n$, but we shall touch on this later in the introduction.

Before motivating the problem we need a couple of definitions.  Given $f:G \rightarrow \R$ we define its \textit{Fourier transform} and \textit{spectral norm} respectively by
\begin{equation*}
\wh{f}(r):=\E_{x \in G}{f(x)(-1)^{r^tx}} \text{ for all }r \in G, \text{ and }\|f\|_{A}:=\sum_r{|\wh{f}(r)|}.
\end{equation*}
It has been known since \cite[Theorem 4.12]{kusman::} that if a \textit{Boolean function} (meaning a function taking only the values $0$ and $1$) has small spectral norm then (an approximation to) it can be easily learnt \cite[p1338]{kusman::}.  (See also \cite{man::}.)  In view of this it is natural to ask how rich the class of Boolean functions with small spectral norm is, and this was part of what motivated the paper \cite{gresan::}.

There are some obvious members: if $V \leq G$ then it is easy to check that $\|1_V\|_{A}=1$ and so if $V_1,\dots,V_L \leq G$ then\footnote{This notation means that there are signs $\varepsilon_i \in \{-1,1\}$ such that $f=\sum_{i=1}^L{\varepsilon_i1_{V_i}}$.} $f:=\sum_{i=1}^L{\pm1_{V_i}}$ is certainly integer-valued and it also has $\|f\|_A \leq L$.  Our aim is to prove the following sort of converse.
\begin{theorem}\label{thm.main}
Suppose that $f$ is integer-valued with $\|f\|_{A} \leq M$.  Then there are subspaces $V_1,\dots,V_L \leq G$ such that
\begin{equation*}
f=\sum_{i=1}^L{\pm 1_{V_i}} \text{ and } L \leq \exp(M^{3+o(1)}).
\end{equation*}
\end{theorem}
This improves on the bound $L \leq \exp(\exp(O(M^4)))$ in \cite[Theorem 1.3]{gresan::}.  On the other hand if $f$ is a sum of $k$ maps of the form $x\mapsto (-1)^{r_i^tx}$ with the $r_i$s independent then $\|f\|_A =O(\sqrt{k})$ so that we certainly need $L=\Omega(M^2)$ above.

The above result includes a structure theorem for Boolean functions (with small spectral norm) as they are functions taking particular integer values, but does not address the question of which linear combinations of indicator functions of subspaces lead to Boolean functions.

Various sub-classes of the Boolean functions with small spectral norm have been studied: for example the symmetric functions in \cite[Theorem 1.1]{adafawhat::}; the low-degree functions in \cite[Lemma 4]{tsawonxie::}; the functions with small Fourier support in \cite[Theorem 1.3]{shptalvol::0}; and the functions in small\footnote{It is perhaps better to say not enormous rather than small here.} ambient group \cite[Theorem 1.2]{shptalvol::0}.

This last result is particularly worth mentioning as the dependence on the size of the ambient group is very mild.
\begin{theorem}[{\cite[Theorem 1.2]{shptalvol::0}}]\label{thm.stv} Suppose that $f$ is Boolean with $\|f\|_{A} \leq M$.  Then there are subspaces $V_1,\dots,V_L \leq G$ such that
\begin{equation*}
f=\sum_{i=1}^L{\pm 1_{V_i}} \text{ and } L \leq \exp(O(M(M+\log\log |G|))).
\end{equation*}
\end{theorem}
This result is stronger than Theorem \ref{thm.main} unless $|G|\geq \exp(\exp(M^{2+o(1)}))$, though it does not cover integer-valued functions more generally.

Statements producing weaker structures (but usually with better quantitative information) have also been proved about Boolean functions with small spectral norm including \cite[Theorem 4]{gro::0}, \cite[Theorems 1.1 \&1.4 ]{shptalvol::0} and \cite[Lemmas 4 \& 8]{tsawonxie::}.

\section{Outline and conditional proof of main theorem}

The overall structure is not wildly different to that in \cite{gresan::}.  We use an induction over almost integer-valued functions, where we say that $f:G \rightarrow \R$ is \textit{$\epsilon$-almost integer-valued} if there is a function $f_\Z$ such that $\|f-f_\Z\|_\infty \leq \epsilon$.  If, as will always be the case, $\epsilon <\frac{1}{2}$ then $f_\Z$ is uniquely determined.

Given a finite non-empty set $S$ in $G$ we write $\mu_S$ for the uniform probability measure on $S$, and given a function $f$ on $G$ we then write
\begin{equation*}
f\ast \mu_S(x):=\E_{x+S}{f}\text{ for all }x\in G.
\end{equation*}
In particular, if $V \leq G$ then $1_A \ast \mu_V(x)$ is the relative density of $A$ on $x+V$.

The idea is to keep splitting $f$ up into pieces with smaller spectral norm, though they may also be less close to being integer-valued.  We do this in two parts: given $f$ we find a subspace with which it correlates.  This is done first by passing to a set with small doubling (Proposition \ref{prop.st} below), and then using a version of Freiman's theorem (Proposition \ref{prop.fre} below).  This result is discussed more in \S\ref{sec.fre}.
\begin{proposition}\label{prop.st}
There is an absolute $C>0$ such that the following holds.  Suppose that $f$ is $\epsilon$-almost integer-valued; and $\|f\|_{A} \leq M$ with $\epsilon \leq \exp(-CM)$.  Then there is a set $A \subset \supp f_\Z$ such that $|A+A| \leq \exp(O(M\log M))|A|$ and $|A| \geq \exp(-O(M\log M))|\supp f_\Z|$.
\end{proposition}
\begin{proposition}\label{prop.fre}
Suppose that $A \subset G$ has $|A+A| \leq K|A|$.  Then there is some $V \leq G$ with $|V| \geq \exp(-\log^{3+o(1)}K)|A|$ and $|A\cap V| \geq \exp(-\log^{1+o(1)}K)|V|$.
\end{proposition}
These are proved in \S\ref{sec.st} and \S\ref{sec.fre} respectively.

For the second part, given this subspace we try to make $f$ behave continuously on it to restore the property of being almost integer-valued.  This is the purpose of Proposition \ref{prop.qc} and is roughly an analogue to \cite[\S3]{gresan::}.
\begin{proposition}\label{prop.qc}
Suppose that $V \leq G$; $\|f\|_A \leq M$; and $\epsilon \in (0,1]$ and $p\geq 2$ are parameters.  Then there is a subspace $U \leq V$ with $\codim_V U =O(p\epsilon^{-2}\log \epsilon^{-1})$ and
\begin{equation*}
\|f - f\ast \mu_U\|_{L_p(\mu_W)}\leq \epsilon M \text{ for all }W \in G/U.
\end{equation*}
\end{proposition}
This is proved in \S\ref{sec.qc}.

With these in hand we produce our key iteration lemma:
\begin{lemma}\label{lem.itlem}
There is an absolute constant $C>0$ such that the following holds.  Suppose that $f$ is $\epsilon$-almost integer-valued; $\|f\|_{A} \leq M$; $\eta>0$ is a parameter and $\epsilon \leq \exp(-CM)$.  Then there is some $V \leq G$ such that $f\ast \mu_V$ is $(\epsilon + \eta)$-almost integer-valued, $(f\ast \mu_V)_\Z \not \equiv 0$ and $|V| \geq  \exp(-O(M^{2+o(1)}\max\{\log \eta^{-1},M\}))|\supp f_\Z|$.
\end{lemma}
\begin{proof}
Apply Proposition \ref{prop.st} (possible provided $\epsilon \leq \exp(-CM)$).  This gives us a set $A \subset \supp f_\Z$ which we can put into Proposition \ref{prop.fre} to get $U \leq G$ such that
\begin{equation*}
|U| \geq \exp(-M^{3+o(1)})|\supp f_\Z| \text{ and }|A\cap U| \geq \exp(-M^{1+o(1)})|U|.
\end{equation*}
Apply Proposition \ref{prop.qc} with a parameter $p$ to be optimised and we get $V \leq G$ with $|V| \geq \exp(-O((p+M)M^{2+o(1)}))|\supp f_\Z|$ such that
\begin{equation*}
\|f - f\ast \mu_V\|_{L_p(\mu_W)} \leq 2^{-4} \text{ for all }W \in G/V.
\end{equation*}
We may assume that $\epsilon \leq 2^{-4}$ and hence by the triangle inequality that
\begin{equation*}
\|f_\Z - f\ast \mu_V\|_{L_p(\mu_W)} \leq \frac{1}{8} \text{ for all }W \in G/V.
\end{equation*}
It follows that $f\ast \mu_V$ (which is constant on cosets of $V$) is certainly $\frac{1}{8}$-almost integer-valued, but crucially this can be bootstrapped.  Suppose $W \in G/V$.  Then
\begin{equation*}
\mu_W\left(\left\{x \in W: (f \ast \mu_V)_\Z(x)\neq f_\Z(x)\right\}\right)\leq \|(f\ast \mu_V)_\Z - f_\Z\|_{L_p(\mu_W)}^p \leq \left(\epsilon + \frac{1}{8}\right)^p\leq 4^{-p}.
\end{equation*}
Writing $W=z+V$ for some $z \in G$ it follows that
\begin{align*}
|f\ast \mu_V(z)-(f \ast \mu_V)_\Z(z)| & \leq |(f - f_\Z) \ast \mu_V(z)| + |(f_\Z-(f\ast \mu_V)_\Z)\ast \mu_V(z)|\\ & \leq \epsilon + O\left(M\mu_W\left(\left\{x \in W: (f \ast \mu_V)_\Z(x)\neq f_\Z(x)\right\}\right)\right).
\end{align*}
We conclude that $f\ast \mu_V$ is $(\epsilon + O(M2^{-p}))$-almost integer-valued from our earlier estimate for the measure.  Finally, if $(f\ast \mu_V)_\Z \equiv 0$ then we have
\begin{equation*}
\mu_W\left(\left\{x \in W: 0\neq f_\Z(x)\right\}\right) \leq 2^{-p} \text{ for all }W \in G/V,
\end{equation*}
but by averaging there is some $W \in G/V$ such that $\mu_W(\supp f_\Z) \geq \mu_W(A) \geq \mu_U(A \cap U)$.  It follows that we may take $p=O(\max\{\log \mu_U(A \cap U)^{-1},\log M\eta^{-1}\})$ such that $f \ast \mu_V$ is $(\epsilon+\eta)$-almost integer-valued, and $2^{-p}<\mu_U(A\cap U)$ from which the lemma follows.
\end{proof}
The result we shall prove (from which Theorem \ref{thm.main} follows immediately) is then the following.
\begin{theorem}\label{thm.det}
There is an absolute constant $C>0$ such that the following holds.  Suppose that $f$ is $\epsilon$-almost integer-valued; $\|f\|_{A} \leq M$; and $\epsilon \leq \exp(-CM)$.  Then there are subspaces $V_1,\dots,V_L \leq G$ such that
\begin{equation*}
f_\Z=\sum_{i=1}^L{\pm 1_{V_i}} \text{ and } L \leq \exp(M^{3+o(1)}).
\end{equation*}
\end{theorem}
\begin{proof}
Let $C>0$ be the absolute constant in the statement of Lemma \ref{lem.itlem}.  Let $\epsilon_i:=2^i\epsilon + 4^{i-2M-4}\exp(-CM)$.  We shall define functions $f_i$ such that
\begin{equation*}
f_i \text{ is $\epsilon_i$-almost integer-valued, } \|f_{i+1}\|_{A} \leq\|f_i\|_A - \frac{1}{2},
\end{equation*}
and so that $(f_i-f_{i+1})_\Z$ can be written as a $\pm1$ sum of at most $\exp(M^{3+o(1)})$ cosets of a subspace $V_{i+1}$.  We set $f_0:=f$ which is certainly $\epsilon_0$-almost integer-valued.  At stage $i \leq 2M+1$ apply Lemma \ref{lem.itlem} with $\eta=4^{-2M-3}\exp(-CM)$ which is possible provided $\epsilon \leq \exp(-C'M)$.  We get $V_{i+1} \leq G$ with
\begin{equation*}
|V_{i+1}| \geq \exp(-M^{3+o(1)})|\supp (f_i)_\Z| \text{ and } f_i \ast \mu_{V_{i+1}} \text{ is }(\epsilon_i+\eta)-\text{almost integer-valued}.
\end{equation*}
Put $f_{i+1}:=f_i - f_i \ast \mu_{V_{i+1}}$.  Then $f_{i+1}$ is $2\epsilon_i + \eta \leq \epsilon_{i+1}$ almost integer-valued.  Moreover, since
\begin{equation*}
|\supp (f_i\ast \mu_{V_{i+1}})_\Z| \leq 2|\supp (f_i)_\Z|
\end{equation*}
and $(f_i\ast \mu_{V_{i+1}})_\Z$ is invariant on cosets of $V_{i+1}$ it follows from the lower bound on $|V_{i+1}|$ that $(f_i\ast \mu_{V_{i+1}})_\Z$ takes non-zero integer values on at most $\exp(M^{3+o(1)})$ translates of $V_{i+1}$.  Added to this, the value of $(f_i\ast \mu_{V_{i+1}})_\Z$ on each of these is an integer between $-(M+1)$ and $(M+1)$.  It follows that $(f_i-f_{i+1})_\Z=(f_i\ast \mu_{V_{i+1}})_\Z$ can be written as a $\pm1$ sum of at most $\exp(M^{3+o(1)})$ cosets of $V_{i+1}$.

Finally, since $(f_i\ast \mu_{V_{i+1}})_\Z$ is not identically $0$ it follows that $\|f_i\ast \mu_{V_{i+1}}\|_A \geq 1-\epsilon_{i+1} \geq \frac{1}{2}$ and hence $ \|f_{i+1}\|_{A} =\|f_i\|_A -\|f_i\ast \mu_{V_{i+1}}\|_A \leq\|f_i\|_A - \frac{1}{2}$.  In view of this the iteration terminates in $2M$ steps and unpacking what that means we have the result.
\end{proof}
The $3+o(1)$ in Theorem \ref{thm.det} arises at two different points.  The first is in the application of Freiman's theorem.  While we do not know how to improve that result, in this case there is a lot more structural information available to us in the proof of Proposition \ref{prop.st} and better bounds can be achieved in this setting (at the expense of the wider applicability of the result).

The second is in Proposition \ref{prop.qc} where, at least for fixed $p$ (\emph{e.g.} $p=2$) it is unclear how to improve the dependencies given the example of $f$ being a sum of maps of the form $x\mapsto (-1)^{r_i^tx}$ where the $r_i$s are independent.

One of the key purposes of this note is to help with the understanding of \cite{san::12}.  Inevitably that paper is rather more complicated but we have followed its overall structure closely.  Roughly speaking Lemma \ref{lem.itlem} corresponds to \cite[Lemma 10.2]{san::12}, Proposition \ref{prop.qc} to \cite[Proposition 7.1]{san::12}, and  Proposition \ref{prop.fre} to \cite[Proposition 8.1]{san::12}; the least similar is Proposition \ref{prop.st} which corresponds to a combination of \cite[Lemma 9.1]{san::12}, \cite[Proposition 9.2]{san::12} (and the Balog-Szemer{\'e}di-Gowers lemma).

\section{Quantitative continuity}\label{sec.qc}

In this section we shall prove the following
\begin{proposition*}[Proposition \ref{prop.qc}]
Suppose that $V \leq G$; $\|f\|_A \leq M$; and $\epsilon \in (0,1]$ and $p\geq 2$ are parameters.  Then there is a subspace $U \leq V$ with $\codim_V U =O(p\epsilon^{-2}\log \epsilon^{-1})$ and
\begin{equation*}
\|f - f\ast \mu_U\|_{L_p(\mu_W)}\leq \epsilon M \text{ for all }W \in G/U.
\end{equation*}
\end{proposition*}
To prove this we shall need the following corollary of work of Croot, {\L}aba and Sisask \cite{croabasis::}.  One could also proceed using Chang's Lemma \cite[Lemma 4.35]{taovu::}.  We have chosen the former approach because it is closer to the argument in \cite[\S7]{san::12}, where the appropriate localisation of Chang's Lemma would be more involved.
\begin{lemma}\label{lem.csl}
Suppose that $V \leq G$, $f \in A$ and $\epsilon \in (0,1]$, $p\geq 2$ are parameters.  Then there is a subspace $U \leq V$ with $\codim_V U = O(p\epsilon^{-2})$ such that
\begin{equation*}
\|f\ast \mu_U-f\|_{L_p(\mu_V)} \leq \epsilon \|f\|_{A}.
\end{equation*}
\end{lemma}
\begin{proof}
Apply \cite[Corollary 3.6]{croabasis::} to $f$ to get some $k = O(p\epsilon^{-2})$, $r_1,\dots,r_k \in G$, and complex numbers of unit modulus $\omega_1,\dots,\omega_k$ such that\footnote{Here $(-1)^{r_i^t\cdot}$ denotes the function $x\mapsto (-1)^{r_i^tx}$.}
\begin{equation*}
\left\|f - \frac{\|f\|_A}{k}\left(\omega_1(-1)^{r_1^t\cdot}+\cdots + \omega_k(-1)^{r_k^t\cdot}\right)\right\|_{L_p(\mu_V)} \leq \epsilon\|f\|_A.
\end{equation*}
Let $U:=V\cap \bigcap_{i=1}^k{\{x:r_i^tx=0\}}$ and note that the given sum is invariant under translation by elements of $U$.  The result follows by the triangle inequality on rescaling $\epsilon$.
\end{proof}
\begin{proof}[Proof of Proposition \ref{prop.qc}]
We produce subspaces $U_i \leq V$ iteratively; initialise with $U_0:=V$.  At stage $i$ suppose that 
\begin{equation}\label{eqn.wantfalse}
\|f - f\ast \mu_{U_i}\|_{L_p(\mu_W)}>  \epsilon \|f\|_{A} \text{ for some }W \in G/U_i.
\end{equation}
By translation we may suppose that $W=U_i$. Apply Lemma \ref{lem.csl} to $f$ with parameter $\epsilon_0:=\frac{1}{2}\epsilon$ to get a space $Z_0\leq U_i$ with
\begin{equation*}
\codim_{U_i}Z_0 = O(\epsilon^{-2}p) \text{ and }\|f\ast \mu_{Z_0} - f\|_{L_p(\mu_U)} \leq \frac{1}{2}\epsilon \|f\|_A.
\end{equation*}
By the triangle inequality we have
\begin{equation}\label{eqn.lower}
\frac{1}{2}\epsilon\|f\|_A < \|f - f\ast \mu_{U_i}\|_{L_p(\mu_{U_i})} - \|f - f\ast \mu_{Z_0}\|_{L_p(\mu_{U_i})} \leq \| f\ast \mu_{U_i}- f\ast \mu_{Z_0}\|_{L_p(\mu_{U_i})}.
\end{equation}
At this point we might use a simply upper bound this last term by the spectral norm and find that we have a large $\ell_1$-mass of $\wh{f}$ on $Z_0^\perp\setminus U_i^\perp$.  This could then be iterated.  We do a little bit better by a dyadic decomposition.\footnote{This is the same idea as is discussed after \cite[Lemma 4.1]{grekon::} where a power of $\frac{1}{4}$ is improved to $\frac{1}{3}$.}

Let $\epsilon_j:=2^{j-1}\epsilon$ and at stage $j>0$ let $Z_j \leq Z_{j+1} \leq U_i$ be a space of minimal codimension such that
\begin{equation*}
\|f\ast \mu_{Z_{j}} - f\ast \mu_{Z_{j+1}}\|_{L_p(\mu_{U_i})} \leq \epsilon_{j+1} \|f\ast \mu_{Z_j^\perp}\|_{A}.
\end{equation*}
By Lemma \ref{lem.csl} applied to $f \ast \mu_{Z_{j}}$ with parameter $\epsilon_{j+1}$ we get a space $Y_{j+1} \leq U_i$ with
\begin{equation*}
\codim_{U_i}Y_{j+1} =O(\epsilon_{j+1}^{-2}p) \text{ and }\|f\ast \mu_{Z_{j}} - f\ast \mu_{Z_j}\ast \mu_{Y_{j+1}}\|_{L_p(\mu_{U_i})} \leq \epsilon_{j+1} \|f\ast \mu_{Z_j^\perp}\|_{A}.
\end{equation*}
It follows that we can take $Z_{j+1}=Z_j+Y_{j+1}$ and have $\codim_{U_i} Z_{j+1} \leq \codim_{U_i} Y_{j+1}$ and $Z_0 \leq Z_1 \leq\cdots $.  On the other hand if $J=\lceil \log\epsilon^{-1}\rceil+2$ then $\epsilon_J \geq 2$ and we can certainly take $Z_J=U_i$ by the triangle inequality.

By the triangle inequality it follows that
\begin{equation*}
\|f\ast \mu_{U_i} - f\ast \mu_{Z_0}\|_{L_p(\mu_{U_i})} \leq \sum_{j=0}^{J}{\|f\ast \mu_{Z_{j}} - f\ast \mu_{Z_{j+1}}\|_{L_p(\mu_{U_i})}} \leq \sum_{j=0}^J{ \epsilon_{j+1} \|f\ast \mu_{Z_j}\|_{A}},
\end{equation*}
and so by averaging and (\ref{eqn.lower}) there is some $j$ such that
\begin{equation*}
\frac{1}{2(J+1)}\epsilon\|f\|_A \leq\epsilon_{j+1} \|f\ast \mu_{Z_j}\|_{A}.
\end{equation*}
But then $\codim_{U_i}Z_j = O(\epsilon_{j+1}^{-2}p)$; set $U_{i+1}:=Z_j$.

Returning to our main iteration, and writing $d_{i+1}:=\codim_{U_i}U_{i+1}$ we then have
\begin{equation*}
1 \leq d_{i+1}=O(\epsilon_{j+1}^{-2}p) \text{ and }\|f\ast \mu_{U_{i+1}} - f \ast \mu_{U_i}\|_A =\Omega \left(\frac{\epsilon \|f\|_A}{\log \epsilon^{-1}}\sqrt{\frac{d_{i+1}}{p}}\right).
\end{equation*}
Since $U_0 \geq U_1 \geq U_2\geq \dots$ we have
\begin{equation*}
\|f\|_A \geq \sum_{i}{\|f\ast \mu_{U_{i+1}} - f \ast \mu_{U_i}\|_A} = \sum_i{\Omega \left(\frac{\epsilon \|f\|_A}{\log \epsilon^{-1}}\sqrt{\frac{d_{i+1}}{p}}\right)},
\end{equation*}
and since $d_i\geq 1$ for all $i>0$ this sum must involve a finite number of summands and the iteration must terminate with the failure of (\ref{eqn.wantfalse}) -- exactly the conclusion we want.  But then
\begin{equation*}
\codim_VU \leq \sum_i{d_i} \leq  \left(\sup_i{\sqrt{d_i}}\right)\sum_i{\sqrt{d_i}} = O(\epsilon^{-1}\sqrt{p})\cdot O(\epsilon^{-1}\sqrt{p}\log \epsilon^{-1}).
\end{equation*}
This gives the result.
\end{proof}

\section{From small Spectral norm to small doubling}\label{sec.st}

Throughout this section it is most natural to use counting measure, and for convenience if $f \in \ell_1(G)$ and $r \in \N$ we write
\begin{equation*}
f^{(r)}(x):=\sum_{x_1+\cdots + x_r=x}{f(x_1)\cdots f(x_r)} \text{ and } f^{(0)}:=1_{\{0_G\}}.
\end{equation*}
The key result of the section is the following.  It essentially uses a refined version of arithmetic connectivity \cite[Definition 5.2]{gresan::0} and is closely related to ideas of M{\'e}la from \cite{mel::}.
\begin{proposition*}[Proposition \ref{prop.st}]
There is an absolute $C>0$ such that the following holds.  Suppose that $f$ is $\epsilon$-almost integer-valued; and $\|f\|_{A} \leq M$ with $\epsilon \leq \exp(-CM)$.  Then there is a set $A \subset \supp f_\Z$ such that $|A+A| \leq \exp(O(M\log M))|A|$ and $|A| \geq \exp(-O(M\log M))|\supp f_\Z|$.
\end{proposition*}
\begin{proof}
Let $R:=\supp f_{\Z}$ and take $l$ and $m$ to be natural numbers to be chosen shortly.  Suppose that there is some $x \in R^m$ such that for any $S \subset [m]$ with $2 \leq |S| \leq l$ odd we have $f_\Z(\sum_{s \in S}{x_s})=0$. For $1 \leq i \leq m$ put $\omega_i:=\sgn(f_{\Z}(x_i))$ and define $h$ by $\wh{h}(r)=\frac{1}{m} \sum_{j=1}^m{\omega_i(-1)^{r^tx_j}}$ so that $\|h\|_{\ell_1(G)} \leq 1$.

Then for every $1 \leq k \leq l$ we have
\begin{align*}
\left|\langle \wh{h}^{2k+1},\wh{f_{\Z}}\rangle_{L_2(\mu_G)}\right| & = m^{-(2k+1)}\left|\E_r{\left(\sum_{i=1}^m{\omega_i(-1)^{r^tx_i}}\right)^{2k+1}\overline{\wh{f_\Z}(r)}}\right|\\
& \leq m^{-(2k+1)} \sum_{\sigma\in [m]^{2k+1}}{|f_{\Z}(x_{\sigma_1}+\cdots + x_{\sigma_{2k+1}})|}\\
& \leq m^{-(2k+1)}\sum_{\substack{\sigma\in [m]^{2k+1}\\ |\{\sigma_i: i \in [m]\}| \leq k+1}}{\|f_{\Z}\|_{\ell_\infty(G)}}\\ & \leq m^{-(2k+1)}\cdot (M+1)\cdot (2k+1) \cdot m \cdot O(mk)^k=Mm^{-k}O(k)^{k+1}.
\end{align*}
Moreover, by Young's inequality $\|h^{(2k+1)}\|_{\ell_1(G)} \leq 1$ and so by Plancherel's theorem we see that
\begin{align*}
\left|\langle \wh{h}^{2k+1},\wh{f_\Z}\rangle_{L_2(\mu_G)} - \langle \wh{h}^{2k+1},\wh{f}\rangle_{L_2(\mu_G)}\right| &  = \left|\langle h^{(2k+1)} ,f_\Z-f\rangle_{\ell_2(G)}\right|  \leq \|f-f_{\Z}\|_{L_\infty(G)} \leq \epsilon
\end{align*}
for all $0 \leq k \leq l$.

Let $P(X)=a_1X+a_3X^3+\dots + a_{2l+1}X^{2l+1}$ be the Chebychev polynomial (of the first kind\footnote{See \cite[S6.10.6]{zwikraros::} for details.}) of degree $2l+1$ so that $a_{2r+1}=(-1)^{l-r}2^{2r+1}\binom{l+r+1}{l-r}$ and $\|P\|_{L_\infty([-1,1])} \leq 1$.
Note that
\begin{equation*}
\left|\langle P(\wh{h}),\wh{f}-\wh{f_\Z}\rangle_{L_2(\mu_G)}\right| \leq\epsilon\sum_{i}{|a_i|} \leq  \epsilon \sum_{k=1}^l{2^{2k+1}\binom{l+k+1}{l-k}} \leq \epsilon \exp(O(l)).
\end{equation*}
But we also have
\begin{align*}
|\langle P(\wh{h}),\wh{f_\Z}\rangle_{L_2(\mu_G)}|
& \geq (2l+1)|\langle \wh{h},\wh{f_{\Z}}\rangle_{L_2(\mu_G)}| - \sum_{k=1}^l{|a_{2k+1}||\langle \wh{h}^{2r+1},\wh{f_{\Z}}\rangle_{L_2(\mu_G)}|}\\
& \geq (2l+1) - M\sum_{k=1}^l{\binom{l+k+1}{l-k}O(k)^{k+1}m^{-k}}\\
& \geq (2l+1)- MO\left(\frac{l^3}{m}\exp(O(l^2/m))\right).
\end{align*}
Since $-1 \leq \wh{h}(r) \leq 1$ we have $|P(\wh{h})| \leq 1$, and so $|\langle P(\wh{h}),\wh{f}\rangle_{L_2(\mu_G)}| \leq  M$.  Combining the results so far using the triangle inequality gives
\begin{equation*}
M \geq (2l+1)- MO\left(\frac{l^3}{m}\exp(O(l^2/m))\right) - \epsilon \exp(O(l)).
\end{equation*}
It follows that if $\epsilon \leq \exp(C_1l)$ for some sufficiently large $C_1>0$ and $m \geq C_2l^3$ for some sufficiently large $C_2>0$ then for $l=C_3M$ we obtain a contradiction, and the supposition at the start of the proof does not hold.  In view of this we see that
\begin{align*}
 |T|^{m} & \leq \sum_{x \in T^{m}}{\sum_{\substack{S \subset [m] \\ 2 \leq |S| \leq 2l+1\\ |S| \equiv 1 \pmod 2}}{1_R\left(\sum_{s \in S}{x_S}\right)}} \\
 & =  \sum_{x \in T^{m}}{\sum_{k=1}^l{\sum_{\substack{S \subset [m] \\ |S|=2k+1}}{1_R\left(\sum_{s \in S}{x_S}\right)}}}= \sum_{r=1}^l{\binom{m}{2k+1}\langle 1_{T}^{(2k+1)},1_R \rangle_{\ell_2(G)}|T|^{m-(2k+1)}}.
\end{align*}
By averaging there is some $1 \leq k \leq l$ such that
\begin{equation*}
1_{T}^{(4k+2)}(0_G)^{1/2} |R|^{1/2} \geq \langle 1_{T}^{(2k+1)},1_R\rangle_{\ell_2(G)} \geq \frac{1}{m\binom{m}{2k+1}}|R|^{2k+1}.
\end{equation*}
Since $1_{T}^{(4k+2)}(0_G) \leq |T|^{4k-2}1_T^{(4)}(0_G)$ we can apply the Balog-Szemer{\'e}di-Gowers \cite[Theorem 2.29]{taovu::} and we are done.
\end{proof}

\section{A variant of {Freiman's} theorem}\label{sec.fre}

In this section we prove the following result which may be of independent interest.
\begin{proposition*}[Proposition \ref{prop.fre}]
Suppose that $A \subset G$ has $|A+A| \leq K|A|$.  Then there is some $V \leq G$ with $|V| \geq \exp(-\log^{3+o(1)}K)|A|$ and $|A\cap V| \geq \exp(-\log^{1+o(1)}K)|V|$.
\end{proposition*}
It may be worth taking a moment to compare this with existing work.  We consider results of the following form: if $|A+A| \leq K|A|$ then there is a subspace $V$ with $\alpha:=|A\cap V|/|V|$ and $\delta:=|V|/|A|$.

\cite[Theorem 1.4]{san::10} gives $\alpha,\delta \geq \exp(-\log^{3+o(1)}K)$ which is not enough for our purposes.  If one wanted $\alpha \geq K^{-O(1)}$ then this is guaranteed by \cite[Theorem A.1]{san::00}) though only with $\delta \geq \exp(-O(\log^4K))$.  Indeed, it actually seems likely that those methods could be used to get $\alpha = (1-\eta)K^{-1}$ and $\delta \geq \exp(-O_\eta(\log^4K))$; this is close to optimal for $\alpha$.

It is a celebrated conjecture of Marton \cite{ruz::01} called the Polynomial Freiman-Ruzsa conjecture \cite[Conjecture 5.34]{taovu::} that we can take $\alpha,\delta \geq K^{-O(1)}$.

\begin{proof}[Proof of Proposition \ref{prop.fre}]
Apply \cite[Proposition 8.5]{san::10} to get $r = \log^{o(1)}K$ and sets $S$ and $T$ with 
\begin{equation*}
2A \subset S \subset 2rA, |S+T| = O(|S|) \text{ and } |T| \geq \exp(-\log^{3+o(1)}K)|S|.
\end{equation*}
By Pl{\"u}nnecke's inequality we have $|2S| \leq \exp(\log^{1+o(1)}K)|S|$.  Apply the Croot-Sisask Lemma (for example, in the form \cite[Proposition 8.3]{san::10}) with a parameter $m$ to be optimised shortly to get a set $X$ with
\begin{equation*}
|X| \geq \exp(-m^2\log^{1+o(1)}K)|T| \text{ and }mX \subset 4S.
\end{equation*}
Let $l$ be a further parameter to be optimised shortly and note by Pl{\"u}nnecke's inequality again that
\begin{align*}
|mlX| \leq |4lS| &\leq \exp(l\log^{1+o(1)}K)|S|\\ & \leq \exp(l\log^{1+o(1)}K + m^2\log^{1+o(1)}K + \log^{3+o(1)}K)|X|.
\end{align*}
Put $3k+1=ml$ and see that for $K$ sufficiently large we can choose $m=\log^{1+o(1)}K$ (and with $m\geq r\log K$) and $l=\log^{2+o(1)}K$ such that $|(3k+1)X| < 2^k|X|$.  We apply Chang's covering lemma \cite[Lemma 5.2]{san::10} to get a set $T\subset X$ of size $k=\log^{3+o(1)}K$ such that $3X \subset \langle T\rangle +2X$.  It follows that $U:=(k+2)X$ is a group and $\mu_U(X) > 2^{-k} \geq \exp(-\log^{3+o(1)}K)$.

So far the argument is essentially the same as the proof of \cite[Proposition 2.5]{san::10}, and if we stopped here then we would have the conclusion but with a weaker bound on the relative density of $A$ in $V$.  We shall bootstrap what we have into the better result by using the argument behind Chang's theorem \cite[Theorem 4.41]{taovu::} for two different sets.  First, note that
\begin{equation*}
\prod_{s=0}^{m-1}{\frac{|X+sX + A|}{|sX+A|}} \leq \frac{|mX+A|}{|A|} \leq \frac{|(8r+1)A|}{|A|} \leq K^{O(r)},
\end{equation*}
and so by the the pigeon-hole principle there is a set $L=A+sX$ for some $0 \leq s \leq m-1$ such that $|X+L| \leq K^{O(rm^{-1})}|L| = O(|L|)$. (Here we use that $m\geq r\log K$.)  Since the cosets of $U$ partition $G$, by averaging there is some translate of $U$, call it $W$, such that $|X+(W\cap L)| =O(|W\cap L|)$; put
\begin{equation*}
B:=L\cap W \text{ and }D:=\frac{|B+X|}{|B|} = O(1).
\end{equation*}
We may translate so that $W=U$ and now work with the Fourier transform on $U$.  By Chang's Lemma \cite[Lemma 4.35]{taovu::} there is a set $\Lambda \subset \wh{U}$ with
\begin{equation*}
|\Lambda|=O(D\log \mu_U(X)^{-1})=O(\log^{3+o(1)}K) \text{ and }\Spec_{\frac{1}{2\sqrt{D}}}(1_X) \subset \langle \Lambda \rangle.
\end{equation*}
It follows that if $x \in \Spec_{\frac{1}{2\sqrt{D}}}(1_X)^\perp$ (\emph{i.e.} $r^tx=0$ for all $r \in \Spec_{\frac{1}{2\sqrt{D}}}(1_X)$) then
\begin{align*}
|1_B \ast 1_B \ast 1_X \ast 1_X (x)-1_B \ast 1_B \ast 1_X \ast 1_X(0_G)| & = \left|\sum_{r }{|\wh{1_B}(r)|^2|\wh{1_X}(r)|^2((-1)^{r^tx}-1)}\right|\\ & \leq 2\sum_{r \not\in\Spec_{\frac{1}{2\sqrt{D}}}(1_X)}{|\wh{1_B}(r)|^2|\wh{1_X}(r)|^2}\\ & \leq \mu_U(B)\frac{1}{2D}\mu_U(X)^2,
\end{align*}
by Parseval's theorem.  On the other hand, by the Cauchy-Schwarz inequality we have
\begin{equation*}
1_B \ast 1_B \ast 1_X \ast 1_X(0_G)\geq \frac{(\mu_U(B)\mu_U(X))^2}{\mu_U(X+B)} \geq \frac{1}{D}\mu_U(B)\mu_U(X)^2.
\end{equation*}
We conclude that $V:=\langle \Lambda \rangle^\perp \subset B+B+X+X \subset 2mX+2A \subset (16r+2)A$.  Finally, we have
\begin{equation*}
\|1_A \ast \mu_V\|_{L_\infty(G)}\mu_G(A) \geq \|1_A \ast \mu_V\|_{L_2(G)}^2 \geq \frac{\mu_G(A)^2}{\mu_G(A+V)} \geq \mu_G(A)K^{-O(r)},
\end{equation*}
by the Cauchy-Schwarz and Pl{\"u}nnecke inequalities.  It follows that $|A\cap (x+V)| \geq K^{O(r)}|V|$ for some $x \in G$ and we get the result by enlarging $V$ if necessary to be the group generated by $x+V$ -- this is at most twice as big.
\end{proof}

\section*{Acknowledgment}

The author should like to thank Ben Green for suggesting that the writing of this paper might make the work of \cite{san::12} more accessible.

\bibliographystyle{halpha}

\bibliography{references}

\begin{thebibliography}{TWXZ13}

\bibitem[AFH12]{adafawhat::}
A.~Ada, O.~Fawzi, and H.~Hatami.
\newblock Spectral norm of symmetric functions.
\newblock In Anupam Gupta, Klaus Jansen, Jos{\'e} Rolim, and Rocco Servedio,
  editors, {\em Approximation, Randomization, and Combinatorial Optimization.
  Algorithms and Techniques}, pages 338--349, Berlin, Heidelberg, 2012.
  Springer Berlin Heidelberg.

\bibitem[C{\L}S11]{croabasis::}
E.~S. Croot, I.~{\L}aba, and O.~Sisask.
\newblock Arithmetic progressions in sumsets and ${L}^p$-almost-periodicity.
\newblock 2011, arXiv:1103.6000.

\bibitem[GK09]{grekon::}
B.~J. Green and S.~V. Konyagin.
\newblock On the {L}ittlewood problem modulo a prime.
\newblock {\em Canad. J. Math.}, 61(1):141--164, 2009.

\bibitem[Gro97]{gro::0}
V.~Grolmusz.
\newblock On the power of circuits with gates of low {$L_1$} norms.
\newblock {\em Theoret. Comput. Sci.}, 188(1-2):117--128, 1997.

\bibitem[GS08a]{gresan::}
B.~J. Green and T.~Sanders.
\newblock Boolean functions with small spectral norm.
\newblock {\em Geom. Funct. Anal.}, 18(1):144--162, 2008, arXiv:math/0605524.

\bibitem[GS08b]{gresan::0}
B.~J. Green and T.~Sanders.
\newblock A quantitative version of the idempotent theorem in harmonic
  analysis.
\newblock {\em Ann. of Math. (2)}, 168(3):1025--1054, 2008, arXiv:math/0611286.

\bibitem[KM93]{kusman::}
{E}. Kushilevitz and {Y}. Mansour.
\newblock Learning decision trees using the {F}ourier spectrum.
\newblock {\em SIAM Journal on Computing}, 22(6):1331--1348, 1993,
  https://doi.org/10.1137/0222080.

\bibitem[Man94]{man::}
Y.~Mansour.
\newblock Learning {B}oolean functions via the {F}ourier transform.
\newblock {\em Theoretical Advances in Neural Computation and Learning}, pages
  391--424, 1994.

\bibitem[M{\'e}l82]{mel::}
J.-F. M{\'e}la.
\newblock Mesures {$\varepsilon $}-idempotentes de norme born\'ee.
\newblock {\em Studia Math.}, 72(2):131--149, 1982.

\bibitem[Ruz99]{ruz::01}
I.~Z. Ruzsa.
\newblock An analog of {F}re{\u\i}man's theorem in groups.
\newblock {\em Ast\'erisque}, (258):xv, 323--326, 1999.
\newblock Structure theory of set addition.

\bibitem[San12]{san::00}
T.~Sanders.
\newblock On the {B}ogolyubov-{R}uzsa lemma.
\newblock {\em Anal. PDE}, 5(3):627--655, 2012, arXiv:1011.0107.

\bibitem[San13]{san::10}
T.~Sanders.
\newblock The structure theory of set addition revisited.
\newblock {\em Bull. Amer. Math. Soc.}, 50:93--127, 2013, arXiv:1212.0458.

\bibitem[San16]{san::12}
T.~Sanders.
\newblock {Bounds in the idempotent theorem}.
\newblock {\em ArXiv e-prints}, October 2016, 1610.07092.

\bibitem[STV14]{shptalvol::0}
A.~Shpilka, A.~Tal, and B.~Volk.
\newblock On the structure of {B}oolean functions with small spectral norm.
\newblock In {\em Proceedings of the 5th Conference on Innovations in
  Theoretical Computer Science}, ITCS '14, pages 37--48, New York, NY, USA,
  2014. ACM.

\bibitem[TV06]{taovu::}
T.~C. Tao and V.~H. Vu.
\newblock {\em Additive combinatorics}, volume 105 of {\em Cambridge Studies in
  Advanced Mathematics}.
\newblock Cambridge University Press, Cambridge, 2006.

\bibitem[TWXZ13]{tsawonxie::}
H.-Y. Tsang, C.~Wong, N.~Xie, and S.~Zhang.
\newblock Fourier sparsity, spectral norm, and the log-rank conjecture.
\newblock In {\em Proceedings of the 2013 IEEE 54th Annual Symposium on
  Foundations of Computer Science}, FOCS '13, pages 658--667, Washington, DC,
  USA, 2013. IEEE Computer Society.

\bibitem[ZKR03]{zwikraros::}
D.~Zwillinger, S.~G. Krantz, and K.~H. Rosen, editors.
\newblock {\em {CRC} standard mathematical tables and formulae}.
\newblock CRC Press, Boca Raton, FL, 31st edition, 2003.

\end{thebibliography}

\end{document}